\documentclass[11pt]{amsart}

\usepackage{amssymb,amsmath,latexsym,amsthm}%or any usual package
\usepackage[english]{babel}

\newtheorem{theorem}{Theorem}[section]

\newtheorem{proposition}{Proposition}[section]

\newtheorem{conjecture}{Conjecture}[section]

\theoremstyle{definition}

\numberwithin{equation}{section}

\newcommand{\ABC}{{abc}}
\newcommand{\XYZ} {{xyz}}
\newcommand{\sP}{{\mathcal P}}
\newcommand{\sF}{{\mathcal F}}
\newcommand{\sS}{{\mathcal S}}
\newcommand{\FF}{{\mathbb F}}
\newcommand{\QQ}{{\mathbb Q}}

\newcommand{\RR}{{\mathbb R}}

\newcommand{\fS}{{\mathfrak S}}
\newcommand{\fm}{{\mathfrak m}}
\newcommand{\fM}{{\mathfrak M}}

\begin{document}

\title[$xyz$ Conjecture]{Smooth Solutions to the $abc$ Equation:\\
The $xyz$ Conjecture}

%first author

\author{\sc Jeffrey C. LAGARIAS}
\address{Jeffrey C. Lagarias\\
University of Michigan\\
Department of Mathematics\\
530 Church Street\\
Ann Arbor, MI 48109-1043, USA}
\email{lagarias@umich.edu}
%\urladdr{http://www.math.lsa.umich.edu/\~lagarias/}

% second author
\author{\sc Kannan SOUNDARARAJAN}
\address{Kannan Soundararajan\\
Department of Mathematics\\
Stanford University\\
Department of Mathematics\\
Stanford, CA 94305-2025,USA}
\email{ksound@stanford.edu}

\maketitle

%\begin{resume}
%Cet article \'etudie les solutions enti\`eres de l'\`equation $abc$ pour
%lesquelles ni $A$, ni $B$, ni $C$ n'ont de grands facteurs premiers. On pose
%$H(A,B,C)=\max(|A|,|B|,|C|)$, et on consid\`ere les solutions primitives
%($\gcd(A,B,C)=1$) n'ayant aucun facteur premier plus grand que
%$(\log H(A,B,C))^{\kappa}$, pour un $\kappa$ fini donn\'e. Nous montrons que
%la Conjecture $abc$ entraine que pour tout $\kappa<1$ l'\'equation n'a qu'un
%nombre fini de solutions primitives. Nous donnons aussi un r\'esultat
%conditionnel, affirmant que l'hypoth\`ese de Riemann g\'en\'eralis\'ee
%(GRH) implique que pour tout $\kappa>8$ l'\'equation $abc$ a un nombre infini de
%solutions primitives. Nous esquissons la preuve de ce dernier r\'esultat.
%\end{resume}

\begin{abstract}
This paper studies  integer
 solutions
 to the $\ABC$ equation $A+B+C=0$
in which none of $A, B, C$ have a large
prime factor. We set 
$H(A, B,C) = \max(|A|, |B|, |C|)$,
and consider primitive solutions (${\rm gcd}(A, B, C)=1$)
having  no prime factor
larger than 
$(\log H(A, B,C))^{\kappa}$, for  a given finite  $\kappa$.
We show that the $\ABC$ Conjecture implies
that for any fixed $\kappa < 1$ the equation has
only finitely many primitive solutions. 
We also discuss a conditional result,
showing that the  Generalized
Riemann hypothesis (GRH) implies 
that for any fixed $\kappa>8$  the $\ABC$ equation
has infinitely many primitive solutions.
 We outline  a proof of the latter  result.

\end{abstract}
\bigskip
%***************************************************
%
%
% Section 1 Introduction
%
%
%***************************************************
\setlength{\baselineskip}{1.0\baselineskip}

\section{Introduction}

The $\ABC$ equation is the 
homogeneous linear ternary Diophantine
equation 
$A+B+C=0,$
usually  written $A+B= C$ (replacing $C$ with $-C$). 
An integer  solution $(A, B, C)$ is 
 {\em nondegenerate} if $ABC \ne 0$,
and  is {\em primitive} if ${\rm gcd}(A, B, C)=1$.
The Diophantine size of a solution can be measured by the 
{\em height} $H(A, B, C)$, given by
\begin{equation}\label{102a}
H(A, B, C) := \max ( |A|, |B|, |C|).
\end{equation}

The well-known $\ABC$ conjecture (of Masser and Oesterl\'{e} \cite{Oe88})
relates the height (\ref{102a})
of solutions to the {\it radical} $R(A,B,C)$ of a
solution, which  is  given by
\begin{equation}\label{102b}
R(A, B, C) = {\rm rad}(ABC) := \prod_{p | ABC} p.
\end{equation}

%*****************************************************************
%
% ABC conjecture-weak
%
%*****************************************************************

{\bf { $\ABC$} conjecture (weak form).}
{\em There is
a positive  constant $\kappa_1$ such that for any 
$\epsilon>0 $ there are only
finitely many primitive solutions $(A,B, C)$
to the $\ABC$ equation $A+B=C$ such that
\[
R(A, B, C)  \le  H(A, B, C)^{\kappa_1 - \epsilon}.
\]
}\\

The  restriction to primitive solutions is needed
to exclude infinite families of imprimitive
solutions  such as $2^n + 2^n = 2^{n+1}$,
which have $R(A, B, C)=2$ and $H(X, Y, Z)=2^{n+1}$.

For any individual
solution $(A, B, C)$ we define  its {\em  $\ABC$-exponent} $\kappa_1(A, B,C)$  by
\[
\kappa_1(A, B, C) := \frac{\log R(A, B, C)}{\log H(A, B, C)}.
\]
Then the  maximum allowable exponent in the $\ABC$ conjecture is given by
\[
\kappa_1 := \liminf_{{H(A,B, C) \to \infty}\atop{A+B=C, gcd(A, B, C)=1}} \kappa_1(A, B, C).
\]
It is known that this exponent satisfies  $\kappa_1 \le 1$, and the $\ABC$ conjecture is 
often stated in the following strong form, e.g. in Bombieri and Gubler \cite[Chap. 12]{BG06}.\\

%*****************************************************************
%
% ABC conjecture-strong
%
%*****************************************************************
%\paragraph

{\bf { $\ABC$} conjecture (strong form).}
 {\em
The $\ABC$ conjecture holds with $\kappa_1=1$, so that for any $\epsilon>0$
there are only finitely many primitive solutions to $A+B=C$ satisfying
\[
R(A, B, C)  \le  H(A, B, C)^{1 - \epsilon}.
\]
}\\

%***************************************************
%
%
% Subsection 1.1 XYZ conejcture
%
%
%***************************************************

\subsection{ ${\bf \XYZ}$ conjecture}

In this paper we relate the height to another measure of solution size,
the {\em smoothness} $S(A, B, C)$, given by
\begin{equation}\label{104}
S(A, B, C) := \max\{p:~  p|ABC  \}.
\end{equation}
We study the existence of primitive solutions
having  minimal smoothness  as a function of the height $H(A,B,C)$.
For convenience we switch variables from $(A, B, C)$
to $(X, Y, Z)$ and  formulate
the following conjecture.\\

%*****************************************************************
%
% Smooth ABC Conjecture// XYZ Conjecture
%
%*****************************************************************
{\bf $\XYZ$ conjecture (weak form-1).}
{\em
 There exists a
positive constant $\kappa_0$ such that the following
hold.

 (a) For
each  $\epsilon >0$ there are only finitely many 
primitive solutions $(X,Y,Z)$ to  
the equation $X+Y=Z$   with
\[
S(X, Y, Z) < (\log H(X, Y, Z))^{\kappa_0 - \epsilon}.
\]

 (b)  For
each $\epsilon >0$ there are infinitely many 
primitive solutions $(X,Y,Z)$ to   to 
the equation $X+Y=Z$ with
\[
S(X, Y, Z) < (\log H(X, Y, Z))^{\kappa_0 + \epsilon}.
\]
}\\

The restriction to primitive solutions is needed, 
to exclude the same infinite family  $2^n+2^n=2^{n+1}$, which has  
$S(X, Y, Z)=2$ and $H(X, Y, Z)=2^{n+1}.$

For any solution $(X, Y, Z)$ 
we define its {\em smoothness exponent} $\kappa_0(X, Y, Z)$ to be 
\begin{equation}\label{109d}
\kappa_0(X, Y, Z) := \frac{\log S(X,Y, Z)} {\log \log H(X, Y, Z)}.
\end{equation}
Define the {\em  $\XYZ$-smoothness exponent}
\begin{equation}\label{109e} 
\kappa_0 := \liminf_{{H(X, Y, Z) \to \infty}\atop{X+Y=Z, gcd(X, Y, Z)=1}} \kappa_0(X, Y, Z).
\end{equation}
{\sl A priori} this exponent exists and
 satisfies $0 \le \kappa_0 \le + \infty $.  \\
 
%*****************************************************************
%
% Smooth ABC Conjecture// XYZ Conjecture-version 2
%
%*****************************************************************

{\bf $\XYZ$ conjecture (weak form-2).}
{\em The $\XYZ$-smoothness exponent $\kappa_0$ is
positive and finite.}\\

To establish the truth of the weak form of the $\XYZ$ conjecture,
it  suffices to show that property $(a)$ holds for some $\kappa_0 >0$
and that property $(b)$ holds for some $\kappa_0 < \infty$. 
The fact that 
properties $(a)$ and $(b)$ are mutually exclusive, 
and are  monotone in $\kappa$ in  appropriate direction, then shows
a unique constant $\kappa_0$ exists satisfying both $(a)$ and $(b)$.
This makes it possible in principle to  prove 
that a nonzero constant $\kappa_0$ exists, without determining
its exact value. 

 %***************************************************
%
%
% Subsection 1.2  Heursitic
%
%
%***************************************************
\subsection{Strong form of  $\XYZ$ conjecture}

There is a simple heuristic argument which supports
the $\XYZ$ -conjecture, and which 
 suggests  that the correct
constant might be $\kappa_0=\frac{3}{2}.$ 
Consider triples $(X,Y,Z)$ where $X$, $Y$ and $Z$ 
are pairwise relatively prime, all lie in the interval $[1,H]$ 
and are all composed of prime factors smaller than $(\log H)^{\kappa}$. 
For each such triple $X+Y-Z$ is an integer in the interval $[-H,2H]$ 
and if $X+Y-Z$ is ``randomly distributed'' then we may expect 
that roughly $1/H$ of the triples $(X,Y,Z)$ will 
satisfy $X+Y=Z$.  We must now count how many triples $(X,Y,Z)$ 
there are and  
expect to find a solution to $X+Y=Z$ when there are 
many more than $H$ such triples, e.g. $H^{1+\epsilon}$ triples,
We also  expect to find no solution when there are far fewer than $H$ such 
triples, i.e. $H^{1- \epsilon}$ triples. 

Let $\sS(y)$ denote the set of all integers having only prime factors $p \le y$,
Recall that 
\begin{equation}\label{121a}
\Psi(x,y) := \#\{n \le x: x \in \sS(y) \}
\end{equation}
counts the number of  such integers below $x$.
It is known that for fixed $\kappa>1$, one has
\begin{equation}\label{109a}
\Psi(x, (\log x)^{\kappa}) = x^{1 - \frac{1}{\kappa} +o(1)},  
\end{equation}
as $x \to \infty$ (see Tenenbaum \cite[Chap. III.5, Theorem 10]{Te95}, \cite[Lemma 9.4]{LS09}).
Thus for $\kappa >1$ the number of such triples $(X,Y,Z)$ with $X, Y, Z \in [1, H]$
is at most $\Psi(H,(\log H)^{\kappa})^3 = H^{3(1-\frac{1}{\kappa}+o(1))}$, 
and if $\kappa < \frac{3}{2}$ this is $< H^{1-\epsilon}$.  Thus 
heuristically we expect few hits in the relatively prime case, 
which suggests that $\kappa_0 \ge \frac 32$.

We now give a lower bound for the number of triples $(X, Y, Z)$. 
Take $X \in [1,H]$ to be a number composed of exactly $K:=[\log H/(\kappa \log \log H)]$ 
distinct primes all below $(\log H)^{\kappa}$.  Using Stirling's formula 
there are \\
%*************************
% forced a line break here
%************************
${\binom{\pi((\log H)^{\kappa})}{K} }= H^{1-1/\kappa+o(1)}$ 
such values of $X$ all lying below $H$. Given $X$,  choose 
$Y \in [1, H]$ to be any number composed of exactly $K$ distinct primes 
below $(\log H)^{\kappa}$,  but avoiding the primes dividing $X$.  
There are ${\binom{\pi((\log H)^{\kappa})-K}{K}} = H^{1-1/\kappa+o(1)}$ 
such values of $Y$.  Finally choose $Z \in [1, H]$ to be a number composed 
of exactly $K$ distinct primes below $(\log H)^{\kappa}$ avoiding 
the primes dividing $X$ and $Y$.  There are 
${\binom{\pi((\log H)^{\kappa})-2K}{K}} = H^{1-1/\kappa+o(1)}$ 
such values of $Z$.  We conclude therefore that 
there are at least $H^{3(1-1/\kappa+o(1))}$ desired triples, 
and hence we expect that $\kappa_0 \le \frac 32$.   

We therefore formulate the following conjecture.\\

%*****************************************************************
%
% Smooth ABC Conjecture// XYZ Conjecture-version 2
%
%*****************************************************************
{\bf $\XYZ$ conjecture (strong form).}
{\em
The $\XYZ$-smoothness exponent $\kappa_0= \frac{3}{2}.$
}\\

In analogy  with the $\ABC$ conjecture, it
  is convenient to  define the {\em $\XYZ$-quality}
 $Q^{\ast}(X, Y, Z)$ of an $\XYZ$ triple $(X, Y, Z)$ to be
\begin{equation}\label{129}
Q^{\ast}(X, Y, Z) := \frac{3}{2} \left(\frac{ \log\log H(X, Y, Z)}{\log S(X, Y, Z)}\right).
\end{equation}
Since
\[
Q^{\ast}(X, Y, Z) =  \frac{3}{2\kappa_0(X, Y, Z)},
\]
the  strong form of the $\XYZ$-conjecture asserts that
\begin{equation}
\label{129b}
\limsup_{H(X,Y,Z) \to \infty} Q^{\ast}(X, Y, Z) = 1.
\end{equation}
Thus triples $(X, Y, Z)$ with $\XYZ$-quality exceeding $1$ are exceptionally
good.  
Unlike the $\ABC$-conjecture case, we do not know if there exist infinitely many relatively prime triples
$(X, Y, Z)$ having quality exceeding $1$.

%***************************************************
%
%
% Subsection 1.2   Examples
%
%
%***************************************************
\subsection{Examples}

 \paragraph{\bf Example 1.}
 Take $(X, Y, Z) = (1, 2400, -2401)$. Then $Y= 2^5 \cdot 3 \cdot 5^2  , |Z|= 7^4$.
 Here the height $H(X, Y, Z) = 2401$ while the smoothness $S(X, Y, Z)=7$.
 We have
 $\log H(X, Y, Z) = 7.78364$ so this example has smoothness exponent
 $$
 \kappa_0(X, Y, Z) := \frac{ \log S(X, Y, Z)}{\log\log H(X, Y, Z)} \approx 0.94829
  $$
 This triple is exceptionally good, having $\XYZ$-quality
 $$
 Q^{\ast} (X, Y, Z):= \frac{3}{2} \left(  \frac{\log\log H(X, Y, Z)}{\log S(X, Y, Z)}
 \right)   \approx 1.58180.
 $$
 
 \paragraph{\bf Example 2.} de Weger \cite[Theorem 5.4]{deW87} found
 the complete set of 
 primitive solutions  to the $xyz$ equation having $S(X, Y, Z) \le 13$;  there are 545  such solutions.  
 His table of large solutions (\cite[Table IX]{deW87}) yields the 
 extremal values  given  in Table \ref{tab1}. His values include the  current record for smallest
 value of  smoothness exponent, which is $\kappa_0 \approx 0.91517$ for
 $(1, 4374, -4375).$\\

\begin{table}
\begin{center}
\begin{tabular}{|c||r|r|r||c|c|}
\hline
\rule[-0.07in]{0cm}{.23in}
S(X,Y,Z) & X & Y & Z & $\kappa_0(X, Y, Z)$ &   $Q^{\ast}(X,Y, Z)$\\
\hline
    3 &  1& 8 & 9 & 1.39560& 1.07480 \\
   5 & 3 & 125 & 128 & 1.01902& 1.47200 \\
  7 & 1 & 4374 & 4375 &   0.91517  & 1.63904\\
  11 & 3584 & 14641 & 18225 & 1.05011 & 1.42841 \\
  13 & 91 & 1771470 & 1771561 &  0.96197 & 1.55930\\
  \hline
\end{tabular}
\end{center}
\caption{Extremal solutions  having $S(X,Y, Z) \le 13$}
\label{tab1}
\end{table}

 \paragraph{\bf Example 3.} 
 Consider the elliptic modular function
  $$
  j(\tau) = \frac{1}{q} = 744 + 196884 q + 21493760 q^2 + ...
  $$
  where $q= e^{2 \pi i \tau}$ with $Im(\tau)>0$ so $|q|<1$. A {\em singular modulus} is
  a value $j (\tau)$ where $\tau$ is an algebraic integer
  in an imaginary quadratic number field $K=\QQ(\sqrt{-d})$; each such $\tau$ corresponds
  to an elliptic curve with complex multiplication by an order in the field $K$.
  The theory of complex multiplication then says that $j(\tau)$ 
  is then an algebraic integer lying in an abelian extension of $K$.
  Gross and Zagier \cite[Theorem 1.3] {GZ85} 
 observe that the norms of differences of singular moduli 
  factorize into products of small primes. 
 Thus, we may consider equations
  \begin{equation}\label{108a}
 \left(j(\tau_1) - j(\tau_2)\right) +  \left(j(\tau_2) - j(\tau_3)\right) + 
 \left(j(\tau_3) - j(\tau_1)\right)=0
 \end{equation}
as giving candidate  triples $(X, Y, Z)$ of algebraic integers $X+Y+Z=0$ that
are smooth in their generating number field $K= \QQ(X,Y, Z)$, in the sense that
the prime ideals dividing $XYZ$ all have small norms.
(These triples were noted in Granville and Stark \cite[Sec. 4.2]{GS00} in
connection with the $\ABC$ conjecture.)
In special cases where all the values $j(\tau_i)$
are rational integers, which occur for example when
each of $\{\tau_k: 1 \le k \le 3\}$  correspond to elliptic curves with CM by an 
 order of  an imaginary quadratic field having class number one,
 then  (\ref{108a})
gives interesting examples of  integer triples with
small smoothness exponents.  
 Using (\cite[p. 193]{GZ85}) we have $j(\frac{-1 + i \sqrt{3}}{2}) =0$, and 
 \[
  j\Big(\frac{1+ i \sqrt{67}}{2}\Big) - j\Big(\frac{-1 + i \sqrt{3}}{2}\Big)= -2^{15}\cdot 3^3 \cdot 5^3 \cdot 11^3
 \]
 \[
~~~~j\Big(\frac{-1 + i \sqrt{3}}{2}\Big)-  j\Big(\frac{1+ i \sqrt{163}}{2}\Big)  = ~2^{18} \cdot 3^3 \cdot 5^3 \cdot 23^3 \cdot 29^3
 \]
 \[
~~~~~~~~~~~~~~~~ j\Big(\frac{1+ i \sqrt{163}}{2}\Big)- j\Big(\frac{1+ i \sqrt{67}}{2}\Big) = -2^{15} \cdot 3^7 \cdot 5^3 \cdot 7^2 \cdot 13
 \cdot 139 \cdot 331.
 \]
  After removing common factors we obtain the relatively prime triple
 \[
   ( -11^3, ~2^3 \cdot 23^3 \cdot 29^3, -3^4 \cdot 7^2 \cdot 13 \cdot 139 \cdot 331)
 =( -1331, ~2 373 927 704, ~-2 373 926 373).
 \]
 This has height $H=2 373 927 704$ and smoothness $S=331$, whence
 \[
  \kappa_0(X, Y, Z) := \frac{ \log S(X, Y, Z)}{\log\log H(X, Y, Z)} \approx 1.88863.
 \]
 Its $\XYZ$-quality $Q^{\ast}(X, Y, Z) \approx 0.79422$.
Here the radical 
$R(X, Y, Z) = 2 \cdot 3 \cdot 7 \cdot 11 \cdot 13 \cdot 23 \cdot 29 \cdot 139 \cdot 331=
184 312 146 018$, which is larger than $H(X, Y,Z)$.\\

\paragraph{\bf Example 4.} Consider the Diophantine equation $x^2+y^3= z^7$.
 Recently Poonen, Schaefer and Stoll \cite{PSS07} proved
 that, up to signs,  it has exactly eight different primitive integer
  solutions, and that
 the  largest of these is $(x, y, z)= (15312283, 9262, 113)$.
 Here $15312283=7\cdot 53\cdot 149\cdot 277$ and $9262= 2 \cdot 11 \cdot 421$.
 Thus for  $(X, Y, Z)= (x^2, y^3, z^7)$ we have
  $H= (113)^7$ and $S=421$, whence
 $\kappa_0(X, Y, Z) \approx 1.72682$ and
 $Q^{\ast}(X, Y, Z) \approx 0.86864$.\\

\paragraph{\bf Example 5.} The Monster simple group has order
$M= 2^{45} \cdot 3^{20} \cdot {5^9} \cdot 7^{6} \cdot 11^2 \cdot 13^3 \cdot 17 \cdot 19 \cdot 23 \cdot
29 \cdot 31 \cdot 41 \cdot 47 \cdot 59 \cdot 71.$
% \approx 8 \cdot 10^{53}$.
Its three smallest irreducible representations are the
identity representation (of degree $1$) plus representations
of degrees $196883= 47 \cdot 59 \cdot 71$ and
$21296876= 2^2 \cdot 31 \cdot 41 \cdot 59 \cdot 71$.
Here we note the triple $(X, Y, Z) =(196882, 1, 196883)$  has  
 $196882= 2 \cdot 7^4 \cdot 41,$ and $S(X, Y, Z)=71$;
additionally the triple $(21296875, 1, 21296876)$ has $21296875=5^6 \cdot 29 \cdot 47,$
and $S(X, Y, Z)=71$. Both these triples have all three terms dividing $M$. 
They have  smoothness exponents 
$\kappa_0(196882, 1, 196883) \approx 1.70463$ 
and $\kappa_0(21296875, 1, 21296876) \approx 1.50849,$ respectively.

%***************************************************
%
%
% Subsection 1.3 Main Results 
%
%
%***************************************************

\subsection{Main results}

The main results presented here are
 conditional results, concerning  both parts of the $\XYZ$ conjecture. 

In \S2 we  show  that a nonzero 
lower bound on $\kappa_0$ follows from the $\ABC$ conjecture. 
Namely, we show the weak form of the $\ABC$ conjecture implies the lower bound $\kappa_0 \ge 
\kappa_1$.

%*****************************************************************
%
% Theorem 1.1
%
%*****************************************************************
\begin{theorem}~\label{th11}
If the  weak form of the $\ABC$-conjecture holds with exponent $\kappa_1$, then for each 
$\epsilon > 0$ there are only finitely many
primitive solutions $(X,Y,Z)$ to $X+Y=Z$ such that
\[
S(X, Y, Z) < \left( \log H(X,Y, Z)\right)^{\kappa_1 - \epsilon}.
\]
Thus the $\XYZ$-smoothness exponent $\kappa_0$ satisfies
$\kappa_0 \ge \kappa_1.$ 
In particular, the strong form of the $\ABC$ conjecture
implies
\[
\kappa_0 \ge 1.
\]
\end{theorem}

The simple deduction of Theorem~\ref{th11} is given in \S2.
We also give an unconditional result (Theorem~\ref{th22}),
namely that there are only finitely many primitive solutions to
\[
S(X, Y, Z) < (3- \epsilon) \log\log H(X, Y, Z).
\]

For the upper bound in the $\XYZ$ conjecture, we derive an upper bound
 in our paper \cite{LS09}, 
assuming appropriate Riemann hypotheses. 
We assume the Generalized Riemann hypothesis (GRH) in the following form:\\

%*****************************************************************
%
% Conjecture (GRH)
%
%*****************************************************************

\noindent{\bf Generalized Riemann Hypothesis ({\rm GRH}).} {\em The Riemann zeta function
and all Dirichlet $L$-functions 
have all of their zeros in the critical strip
$0 < Re(s) <1$ lying on the critical line $Re(s)= \frac{1}{2}.$ }\\

The main result of \cite{LS09}
is  that the GRH implies  the upper bound part of the $\XYZ$-conjecture, establishing that
$\kappa_0 \le 8$.

%*****************************************************************
%
% Theorem 1.2
%
%*****************************************************************
\begin{theorem}~\label{th12} 
If the Generalized Riemann Hypothesis (GRH) holds,  then for each 
$\epsilon > 0$ there are infinitely many primitive
solutions $(X,Y,Z)$ to $X+Y+Z=0$ such that
\[
S(X, Y, Z) < (\log H(X,Y, Z))^{8 + \epsilon}.
\]
In particular, the $\XYZ$-smoothness exponent satisfies $\kappa_0 \le 8$.
\end{theorem} 

We discuss this result and its proof in \S3 and \S4.
Theorem~\ref{th12} is a consequence of  a  stronger result, stated  
as Theorem~\ref{th13} below, 
which gives a lower bound for  the number of weighted primitive 
solutions to the equation, for fixed $\kappa >8$, which is the correct
order of magnitude according to the heuristic argument above.  This result is proved by
a variant of the Hardy-Littlewood method, combined with
the Hildebrand-Tenenbaum saddle-point method for counting numbers
all of whose prime factors are small. The generalized Riemann hypothesis
is invoked  mainly to control the minor arcs estimates.  The Hardy-Littlewood
method detects all integer solutions, and an inclusion-exclusion argument
is needed at the end of the proof  to count primitive integer solutions only.

Combining Theorems \ref{th11} and \ref{th12} 
above yields the following result conditionally proving the $\XYZ$ conjecture.

%*****************************************************************
% [Theorem 1.3 here]
%
%*****************************************************************
\begin{theorem}~\label{th14}{\em (Alphabet Soup Theorem)}
 The weak form of the $\ABC$-conjecture  together with the $GRH$  implies the 
 weak form of the $\XYZ$ conjecture.
 \end{theorem}

This conditional result does not establish the strong form of the $\XYZ$ conjecture, since it only shows
 $\kappa_1 \le \kappa_0 \le 8$.  Even if we assume
  the strong form of the $\ABC$ conjecture, we deduce only 
 $\kappa_0 \ge 1$ which is weaker than the lower bound of the strong $\XYZ$ conjecture.

%***************************************************
%
%
% Section 1. 4
%
%
%***************************************************
\subsection{$S$-unit equations}

The $S$-unit equation is 
$X+Y=Z$,
subject to the constaint that all prime factors of $XYZ$ belong to
a fixed finite set of primes $S$. 
In 1988 Erd\H{o}s, Stewart and Tijdeman \cite{EST88} showed the existence
of collections of primes $S$ with $|S|=s$ such
that the $S$-unit equation $X+Y= Z$ has ``exponentially many''
solutions, namely at least 
$\exp\left((4- \epsilon) s^{\frac{1}{2}} (\log s)^{-\frac{1}{2}} \right)$
solutions, for $s \ge s_0(\epsilon)$ sufficiently large.
Recently Konyagin and Soundararajan \cite{KS07}
improved this construction, to show that there exist $S$ such
that the 
$S$-unit equation has at least
$\exp\left(s^{2- \sqrt{2}- \epsilon}\right)$ solutions.\\
In these constructions   the sets of primes  $S$ were tailored to have 
large numbers of solutions, which entailed losing control over
the relative sizes of primes in $S$.  
However it is natural to consider the special case where
$S$ is an inital segment of primes 
\begin{equation}\label{142}
S_y =\sP(y) :=\{ p: p ~\mbox{prime},~p \le y \}.
\end{equation}
Here
Erd\H{o}s, Stewart and Tijdeman conjectured (\cite[p. 49, top]{EST88})
a stronger exponential bound, which asserts that  
 for $s= |S_y|$ and each $\epsilon>0$
there should be at least $\exp( s^{\frac{2}{3} - \epsilon})$
$S$-unit solutions to $X+Y=Z$ and at most $\exp( s^{\frac{2}{3} + \epsilon})$
such solutions, for all $s > s_0(\epsilon).$ \\

The $xyz$ conjecture can be reformulated as an assertion about the
maximal relatively prime solutions to an $S$-unit equation for $S= S_y$,
as $y \to \infty$.\\

%*****************************************************************
%
% Smooth ABC Conjecture// XYZ Conjecture version3
%
%*****************************************************************

{\bf $\XYZ$ conjecture (weak form-3).}
{\em Take $S=S_y$ 
and  let
$H_y$ denote the maximal height of any relatively prime
solution to this $S$-unit equation.
Then the constant
\begin{equation}\label{148}
\kappa_0 :=\liminf_{y \to \infty}~ \frac{ \log y}{\log\log H_y}.
\end{equation}
is positive and finite. 
}\\

It is known that the complete set
of solutions to an $S$-unit equation for fixed $S$ can be determined effectively,
and therefore $H_y$ is effectively computable in principle
(but only for small $y$ in practice).

The conjecture of Erd\H{o}s, Stewart and Tijdeman for the
number of solutions to the $S$-unit equation for $S=S_y$
would imply  that $\kappa_0 \ge \frac{3}{2}.$ 
In fact  the exponent $\frac{2}{3}$ conjectured by Erd\H{o}s, Stewart and  Tijdeman \cite{EST88} 
is supported by a similar  heuristic to that giving
$\kappa_0= \frac{3}{2}$ in the strong form of the $\XYZ$ conjecture.

We obtain the following result in 
the direction of the conjecture of Erd\H{o}s, Stewart and Tijdeman,
as an easy
consequence of the quantitative bound given in 
Theorem~\ref{th13}  later in this paper.
%*****************************************************************
% [Theorem 1.4 here]
%
%*****************************************************************
\begin{theorem}~\label{th15}
Let $S= S_s^{\ast}$ denote the set of the first $s$ primes, and
let $N(S)$ count the number of 
primitive  solutions  $(X,Y, Z)$ 
to the $S$-unit equation $X+Y=Z.$
If the Generalized Riemann Hypothesis
 is valid, then for each $\epsilon>0$, 
\[
N(S_s^{\ast}) \gg_{\epsilon}  \exp(s^{\frac{1}{8}-\epsilon}).
\]
\end{theorem}

This is proved as Theorem 1.3 in \cite{LS09}.

%***************************************************
%
%
% Section 1. 5 Discussion
%
%
%***************************************************
\subsection{Discussion}~\label{sec15}

We  remark first on  earlier  work related to the approach used in Theorem~\ref{th12}.
The Hardy-Littlewood method was used 
in 1984 by Balog and Sarkozy \cite{BS84a}, \cite{BS84b}
to show that  the equation $X+Y+Z=N$ has solutions with 
 $S(X, Y, Z)  \le \exp( 3 \sqrt{\log N \log\log N})$ for all sufficiently large $N$.
Exponential sums taken over smooth numbers
below $x$ with all prime factors smaller than $y$
% (see \eqn{305} below) 
were  studied in detail by  de la Bret\`{e}che (\cite{dlB98}, \cite{dlB99})
and de la Bret\`{e}che and Tenenbaum (\cite{BT04}, \cite{BT05}, \cite{BT07}),
see also de la Bret\`{e}che and Granville \cite{dBG09}.
Their various bounds are unconditional, and are 
valid  for sufficiently large $y$, requiring at least  $y \ge  \exp (c (\log\log x)^2)$.
Thus their results fail to apply in  the range we need, which is 
$y= (\log x)^{\kappa}$ for fixed $\kappa$.

In the range of interest,   it is
a delicate problem even to count the number of $y$-smooth integers up to $x$.
General results on $\Psi(x,y)$  are reviewed in Hildebrand and Tenenbaum \cite{HT93}.
In \cite{LS09}
we  make use of the saddle-point method developed by Hildebrand and Tenenbaum \cite{HT86}
to estimate $\Psi(x, y)$. 
They showed that it  provides an asymptotic formulas for $\Psi(x,y)$
when $y$ is not too small, including a result of Hildebrand that for $y \ge \exp( (\log \log x)^{\frac{5}{3} + \epsilon})$,
\begin{equation}\label{121c}
\Psi(x, y) =  x\rho(u)\Big( 1+ O_{\epsilon}\Big( \frac{u\log(u+1)}{\log x}\Big) \Big).
\end{equation}
where $u$ is defined by $y= x^{\frac{1}{u}}$ and $\rho(u)$ is Dickman's function.
The range $y= (\log x)^{\kappa}$ for fixed $\kappa$  lies outside the 
range covered by Hildebrand's (\ref{121c}), and in  this range, the 
behavior of $\Psi(x,y)$ is known to be sensitive to the fine distribution of primes and
 location of the zeros of $\zeta(s)$.
In 1984 Hildebrand \cite{Hi84} showed that the Riemann hypothesis is equivalent
to the assertion that for each $\epsilon >0$ and $1 \le u \le y^{1/2 - \epsilon}$
there is a uniform estimate
\[
\Psi(x,y) =   x\rho(u)\exp( O_{\epsilon} ( y^{\epsilon} ) ).
\]
Moreover, assuming the Riemann hypothesis, he showed that 
for each $\epsilon >0$ and $1 \le u \le y^{1/2 - \epsilon}$
the stronger uniform estimate
\[
\Psi(x,y) = x \rho(u) \exp \Big( O_{\epsilon} \Big( \frac{\log (u+1)}{\log y} \Big) \Big)
\]
holds. 
On choosing $y = (\log x)^{\kappa}$ for $\kappa > 2$, this latter estimate yields 
\[
\Psi(x, (\log x)^{\kappa} ) \asymp x \rho(u),
\]
which  provides only an order of magnitude estimate for the size of $\Psi(x,y)$. 
Furthermore if
the Riemann hypothesis is false then $\Psi(x,y)$ must sometimes
exhibit large oscillations away from the value $x\rho(u)$ for some $(x,y)$ in
these ranges.  In 1986 Hildebrand \cite{Hi86} obtained further results 
indicating  that when
$y < (\log x)^{2- \epsilon}$ one should not expect  any smooth asymptotic
formula for $\Psi(x,y)$ in terms of the $y$-variable to hold.
We assume GRH, and because of these issues  we retain $\Psi(x,y)$ in our final formulas, avoiding any
asymptotic approximation to it (see Theorem~\ref{th52} and \ref{th121}). 
Note that the threshold value $\kappa=2$ is relevant in
the formula in  Theorem~\ref{th13},
see (\ref{335}) below.

Concerning the relative strength of the $\ABC$
conjecture versus the $\XYZ$ conjecture,  the difficulty of the $abc$ conjecture lies
entirely in the lower bound. Indeed it has the trivial upper bound $\kappa_1 \le 3$,
and also has the unconditional upper bound $\kappa_1 \le 1$.
The  $xyz$ conjecture presents
difficulties in both  the upper bound and the lower bound,
and we have no unconditional result for either bound.
Concerning the lower bound, we do
not know any implication from the $\XYZ$ conjecture in the direction
of the $\ABC $ conjecture. 
Thus establishing a nontrivial lower bound (a) of the 
$\XYZ$ conjecture could be an easier problem than the $\ABC$-conjecture. 
We note the coincidence that the  $\ABC$ implication that $\kappa_0 \ge1$  
lower bound occurs exactly at
the threshold value of $\kappa=1$ at which the behavior of the 
smooth number counting function $\Psi(x, x^{\kappa})$ qualitatively changes.
For $\kappa >1$ one has $\Psi(x, x^{\kappa}) = x^{1- \frac{1}{\kappa} +o(1)},$
while for  $0<\kappa <1$ one has $\Psi(x, x^{\kappa}) = x^{o(1)}.$

 It remains an interesting question 
 whether the GRH assumption can be removed
to  obtain an unconditional upper bound on $\kappa_0$,
possibly worse than the conditional upper bound $\kappa_0 \le 8$ obtained in \cite{LS09}.
This problem does not seem entirely out of reach.
We are not aware of any approach that might yield a
positive unconditional lower bound for $\kappa_0$.

The plan of the paper is as follows.
In Section 2 we establish the lower bound for the $\XYZ$ constant $\kappa_0$
assuming the $\ABC$ conjecture. In Section 3 we discuss the main upper  bound
obtained in \cite{LS09}
for the number of smooth solutions when $\kappa_0 > 8$, given as Theorem~\ref{th13}, 
and discuss conjectures concerning asymptotics of the number of primitive solutions
versus all solutions for $\kappa>2.$
 Section 4  presents results  from \cite{LS09}
giving asymptotic formulas 
counting weighted smooth solutions, from which Theorem~\ref{th13}
is derived, and in Section 5 we  give some ideas of their proofs.
In Section 6 we make concluding remarks, indicating  directions for extending the
results.\\

\paragraph{\bf Acknowledgments.}  The paper is based on a talk of
the first author
at the 26-th Journ\'{e}es Arithm\'{e}tiques in
St. Etienne, July 2009. We thank Peter Hegarty
and Michael Waldschmidt  raising some
interesting questions there about these results. We thank B. M. M. de Weger
for comments and references on  solutions to these equations.
We  thank the reviewer for helpful comments. During work on
this project the first author received support  
from  NSF grants DMS-0500555 and DMS-0801029, and
the second author from NSF grants DMS-0500711 and DMS-1001068.
Some of this work was done 
while the first author visited Stanford University,
and he thanks the Mathematical Research Center at Stanford for support.

%***************************************************
%
%
% Section 2
%
%
%***************************************************
\section{Lower Bound Assuming the $\ABC$ Conjecture}

The $\ABC$ conjecture was formulated by Masser and Oesterl\'{e}
in 1985, cf. Stewart and Tijdeman \cite{ST86}, Oesterl'{e} \cite{Oe88}. It has many equivalent formulations,
see Bombieri and Gubler \cite[Chap.XII]{BG06}.

%*****************************************************************
%
% Theorem  2.1
%
%*****************************************************************
\begin{theorem}\label{th21} 
Assume the weak form of the $\ABC$ conjecture holds.
Then for any $\epsilon >0$ there are only finitely
many primitive solutions 
 $(X,Y, Z)$
to the equation $X+Y=Z$ such that
\[
S(X,Y, Z)  \le  (\kappa_1-\epsilon) \log H(X, Y, Z).
\]
In particular the $\XYZ$-smoothness exponent $\kappa_0$
satisfies $\kappa_0 \ge \kappa_1$.
\end{theorem}

\begin{proof}
Let $X+Y=Z$.
The radical $R(X,Y,Z)$ satisfies
\begin{equation}\label{211}
R(X, Y, Z) = {\rm rad}(XYZ)= \prod_{p | XYZ} p \le \prod_{p \le S(X,Y,Z)}p
\end{equation}
Now the prime number theorem implies that
\[
\prod_{p \le y} p = e^{y +o(y)} ~~\mbox{as}~~ y \to \infty,
\]
so that
\[
R(X,Y, Z) \le e^{S(X,Y,Z) (1+o(1))}
\]
%\eeq
The hypothesis that a solution has smoothness satisfying
\[
S(X,Y,Z) \le (\kappa_1- \epsilon) \log H(X, Y, Z)
\]
now yields
\begin{equation}\label{214}
R(X, Y, Z) \le H(X,Y, Z)^{\kappa_1 - \epsilon + o(1)} ~~\mbox{as}~~ H \to \infty.
\end{equation}
Thus  such a solution must satisfy
\begin{equation}\label{214a}
R(X, Y, Z) \le H(X,Y, Z)^{\kappa_1 -\frac{1}{2}\epsilon},
\end{equation}
if its height is sufficiently large, say $H\ge H_0$.
The weak  form of  the $\ABC$ conjecture asserts that the inequality (\ref{214a})
has only finitely many primitive integer solutions $(X, Y, Z)$,
so the result follows.
\end{proof}

\begin{proof}[Proof of Theorem~\ref{th11}.]
This follows immediately
from Theorem~\ref{th21}.
\end{proof}

Following the same proof as  Theorem~\ref{th21} above,
but employing the current best bound
on the $\ABC$ equation, yields
the following 
unconditional result.
%*****************************************************************
%
% Theorem 2.2/formerly 2.1
%
%*****************************************************************
\begin{theorem}~\label{th22}
For each 
$\epsilon > 0$ there are only finitely many
primitive solutions $(X,Y,Z)$ to $X+Y+Z=0$ such that
\begin{equation}\label{220}
S(X, Y, Z) \le (3 - \epsilon)  \log \log H(X,Y, Z).
\end{equation}
\end{theorem}

\paragraph{Proof.}
The best unconditional
bound in the direction of the $\ABC$ conjecture
is currently that of  Stewart and Yu \cite{SY01}.
They showed there is a constant $c_1$
such that all primitive solutions to the
$abc$ equation satisfy 
\[
H(X, Y, Z) \le \exp( c_1 R^{\frac{1}{3}} (\log R)^3),
\]
where $R=R(X, Y, Z)$. Taking logarithms yields the constraint
\begin{equation}\label{222}
(\log H(X, Y, Z) )^3 \le   (c_1)^3 R(X,Y,Z) (\log R(X,Y,Z))^9 .
\end{equation}

Suppose that an
 $\ABC$ solution $(X, Y, Z)$ satisfies 
the bound (\ref{220}). 
There are only finitely many  $\ABC$ solutions $(X, Y, Z)$
having $S=S(X,Y,Z)$ taking a given value, since the $\ABC$-equation is then
an $S$-unit equation, with $S$ being all primes up to the given value.
Thus 
by excluding a finite set  of solutions we may assume
$S \ge S_0$ for any fixed $S_0$.
If  $S=S(X, Y, Z)$ is sufficiently large
(depending on $\epsilon$) then  the $o(1)$ term in  (\ref{214})
can be replaced by $\frac{1}{6} \epsilon$ and we obtain 
\[
R(X, Y, Z) \le e^{S(X,Y,Z)(1+\frac{1}{6}\epsilon)} \le 
e^{(3- \frac{1}{2}\epsilon) \log \log H}
=(\log H)^{3- \frac{1}{2} \epsilon}.
\]
Substituting this bound in (\ref{222}) shows this
solution must satisfy
\[
 (\log H)^3 \le  
(c_1)^3 (\log H)^{3- \frac{1}{2} \epsilon}(3\log\log H)^9.
\]
However for all sufficiently large $H$, the right side of this inequality  is 
smaller than $(\log H)^3$,
giving a contradiction.
 We conclude that all such solutions have bounded $H$, and  
the finiteness result follows. $~~~\Box$.

%***************************************************
%
%
% Section 3 
%
%
%***************************************************

\section{Upper Bound Assuming ${\rm GRH}$}

Let $\sS(y)$ denote the set of all integers 
having all prime factors $\le y$. We define the counting function $N(H, \kappa)$ of
all smooth solutions $X+Y=Z$ to height $H$ by 
\[
N(H, \kappa) := \# \{ X+Y=Z: ~ 0 \le X, Y, Z \le H, ~~\max_{p |XYZ}\{p\} \le (\log H)^{\kappa}\},~~~~~~~~~~~~
\]
and the  counting function $N^{\ast}(H, \kappa)$ of all primitive integer solutions to height $H$ by
\[
N^{\ast}(H, \kappa) := 
\# \{ X+Y=Z  ~\mbox{primitive} :  (X,Y, Z) \in N(H, \kappa)\}.
%0 \le X, Y, Z \le H, 
%~\max_{p| XYZ} \{p\}  \le (\log H)^{\kappa}\}.
\]
A main result of  our paper \cite{LS09} 
is the following  lower bound for 
the number of primitive  integer solutions 
for $\kappa >8$ (\cite[Theorem 1.3]{LS09}). 
 
%*****************************************************************
% Theorem 3.1
%
% [Theorem 1.3 of [PLMS]
%
%*****************************************************************

\begin{theorem}~\label{th13}
{\em (Counting Primitive Smooth Solutions)}
Assume the truth of the Generalized Riemann Hypothesis (GRH). Then for each fixed
$\kappa > 8$, the number of primitive integer solutions to $X+Y=Z$ below $H$
satisfies
\begin{equation}\label{131}
N^{\ast}(H, \kappa) \ge 
{\fS}_{\infty}\Big(1-\frac{1}{\kappa}\Big) 
{\fS}_{f}^{\ast}\Big(1-\frac{1}{\kappa}, (\log H)^{\kappa} \Big) 
\frac{\Psi(H, (\log H)^{\kappa})^3}{H}(1 + o(1)),
\end{equation}
as $H \to \infty$. 
Here the ``archimedean singular series" (more properly,  ``singular integral")
is defined, for $c > \frac{1}{3}$, by
\begin{equation}\label{132}
{\fS}_{\infty}(c):= c^3 \int_{0}^{1} \int_{0}^{1-t_1} 
(t_1 t_2 (t_1+t_2))^{c-1} dt_1 dt_2,
\end{equation}
and the ``primitive non-archimedean singular series" ${\fS}_{f}^{\ast}(c, y)$ is defined by
\begin{equation}\label{133}
{\fS}_{f}^{\ast}(c, y):= \prod_{p \le y} \Big( 1 + \frac{1}{p^{3c-1}}
\Big( \frac{p}{p-1}\Big(\frac{p-p^c}{p-1}\Big)^3 - 1\Big)\Big)\prod_{p>y}\Big( 1- \frac{1}{(p-1)^2}\Big).
\end{equation}
\end{theorem}

This result is
derived using a variant of the Hardy-Littlewood method together
with the Hildebrand-Tenenbaum saddle point method for
counting smooth numbers. The method is first applied to 
integer solutions weighted using a smooth compactly-supported
test function in the open interval $\RR_{>0} = (0, \infty)$,
where one can obtain asymptotic formulas, described in Section 4,
for both integer solutions and for primitive integer solutions.
We  obtain  only a
one-sided inequality in  (\ref{131}) because  the weight
function corresponding to (\ref{131}) is not compactly supported,
but can be approximated from below by such functions to give a lower bound.

We discuss the terms appearing on the right side of  this formula.
Since  for $\kappa>1$, we have 
$\Psi(H, (\log H)^{\kappa}) = H^{ 1- \frac{1}{\kappa} + o(1)}$ as $H \to \infty$,
we see that the ``main term" on the right side of the formula (\ref{131}) is essentially of the form 
$H^{2-\frac{3}{\kappa}+o(1)}$ given
by the heuristic of section 1.2, up to the value of the
constants in the ``singular series''. 
We believe  that equality should actually hold in   (\ref{131}) 
in the indicated range of $H$ and $\kappa$, thus giving an asymptotic formula for $\kappa> 8$.

Now we consider the ``main term" on the right side of (\ref{131}) for smaller values of $\kappa$.
It is well-defined   in  the range $\kappa> \frac{3}{2}$ where the heuristic in section 1.2 is expected to apply.
 Here $\kappa> \frac{3}{2}$ corresponds to $c> \frac{1}{3}$,
 and the  ``archimedean singular integral"  (\ref{132})
 defines  an analytic function on the half-plane $Re(c)> \frac{1}{3}$ which
 diverges at $c=\frac{1}{3}$, while 
 the   ``non-archimedean singular series"   $\fS_{f}(c, y)$ is well-defined for all $c>0$.
 The archimedean singular series  is uniformly
 bounded on any half-plane $Re(c)> \frac{1}{3} +\epsilon.$ 
 For the non-archimedean singular series, we find that  its limiting
  behavior as $y= (\log H)^{\kappa} \to \infty$  changes at the threshold value $\kappa=2$, 
  corresponding to $c = \frac{1}{2}$. Namely, one has
\begin{equation}\label{335}
 \lim_{H \to \infty} {\fS}_{f}^{\ast}\Big(1-\frac{1}{\kappa}, (\log H)^{\kappa}\Big) = 
 \left\{ 
\begin{array}{ll} {\fS}_{f}^{\ast}(1-\frac{1}{\kappa}) &~~\mbox{for}~~ \kappa>2,\\
~&~\\
0 &~~\mbox{for} ~~ 0<\kappa \le 2,
\end{array}
\right.
\end{equation}
 where for $c> \frac{1}{2}$ we set
  \begin{equation}\label{321a}
 {\fS}_{f}^{\ast}(c) := \prod_{p} \Big( 1 + \frac{1}{p^{3c-1}}
\Big( \frac{p}{p-1}\Big(\frac{p-p^c}{p-1}\Big)^3 - 1\Big)\Big).
 \end{equation}
  The Euler product  (\ref{321a}) converges absolutely
 and defines an analytic function $ {\fS}_{f}^{\ast}(c)$
 on the
 half-plane $Re(c) > \frac{1}{2}$; 
 this function is uniformly bounded on any
 half-plane $Re(c)> \frac{1}{2}+\epsilon$, 
 Furthermore for values  corresponding the
 $2 < \kappa < \infty$ (i.e. $\frac{1}{2} < c < 1$) the ``non-archimedean
 singular series"  $\fS_{f}^{*}(c, y)$
  remains bounded away from $0$. We conclude that for $2 < \kappa < \infty$
 the ``main term" estimate for $N^{\ast}(H, \kappa)$
 agrees with the prediction of the heuristic argument given earlier.
  In the remaining  region $1 < \kappa\le 2$ 
 one can show that although ${\fS}_{f}^{\ast}(1-\frac{1}{\kappa}, (\log H)^{\kappa}) \to 0$
 as $H \to \infty$, one has
  \[
  {\fS}_{f}^{\ast}\Big(1-\frac{1}{\kappa}, (\log H)^{\kappa}\Big) \gg \exp \Big( -(\log H)^{2-\kappa}\Big). 
  \]
    This bound implies that ${\fS}_{f}^{\ast}(1-\frac{1}{\kappa}, (\log H)^{\kappa}) \gg H^{-\epsilon}$
for any $\epsilon>0$, which for $\frac{3}{2} < \kappa \le 2$ 
shows the ``main term" is still of the same order $H^{2- \frac{3}{\kappa} +o(1)}$ as the
heuristic predicts.
  Thus  it could still  be the case that the ``main term" on the right side of (\ref{132})
 gives a correct order of magnitude estimate  for $N^{\ast}(H, \kappa)$ in this range.
 
 In terms of limiting behavior as $H \to \infty$
 there appears to be  a second threshold value at $\kappa =3$,
 which concerns the relative density of primitive smooth solutions in 
 the set of all smooth solutions. We  formulate this as the following conjecture.
 
 %*****************************************************************
% 
%  Conjecture 3.2
%
%*****************************************************************

\begin{conjecture}\label{cj32}
{\em (Relative Density Conjecture)}
The relative density of primitive smooth solutions satisfies
\begin{equation}\label{370}
 \lim_{H \to \infty} \frac{ N^{\ast}(H, \kappa)}{N(H, \kappa)}= 
 \left\{ 
\begin{array}{ll}\frac{1}{\zeta(2- \frac{3}{\kappa})},& ~~~\mbox{for}~~ 3 < \kappa < \infty,\\
~&~\\
0 &~~\mbox{for} ~~ 1<\kappa \le 3.
\end{array}
\right.
\end{equation}
\end{conjecture}
 
 Assuming GRH, our paper \cite{LS09} shows that a weighted version of this conjecture holds 
 for $\kappa>8$, given here as Theorem \ref{th35}.
  Furthermore  for each $\kappa > 3$ 
  the ratios of the conjectured ``main terms" in the 
 asymptotic formulas for these quantities have this limiting value, a
 result implied by (\ref{423}) below.

 We remark that the number $N(H, \kappa)$ of all smooth integer solutions already has a 
 contribution from smooth multiples of the solution $(X, Y, Z) =(1, 1, 2)$ that gives
 \begin{equation}\label{375}
 N(H, \kappa) \ge \Psi( \tfrac{1}{2} H, (\log H)^{\kappa}) \ge H^{1- \frac{1}{\kappa} +o(1)}, ~~\mbox{as}~~H \to \infty.
 \end{equation}
 For  $1 \le \kappa <2$ 
 this  lower bound exceeds the heuristic argument  estimate $H^{2- \frac{3}{\alpha} +o(1)}$
 for  $N^{\ast}(H, \kappa)$  
    by a positive power of $H$. 
  This  fact indicates that the heuristic in Section 1.2 fails for $N(H, \kappa)$ for $1 < \kappa < 2$,
 and supports the truth of Conjecture \ref{cj32} in this  range.

%***************************************************
%
%  Section 4)
%
%
%***************************************************
%

\section{Asymptotic Formulas for Weighted Smooth Solutions}

We now describe results giving  asymptotic
formulas that count weighted smooth solutions, when $\kappa> 8$.
Let $\Phi(x)\in C_{c}^{\infty}(\RR_{>0})$ be a smooth
compactly supported (real-valued) function on the positive real
axis.
We define the  weighted sum
\begin{equation}\label{304}
N(x, y; \Phi) := \sum_{{X, Y, Z \in  \sS(y)}\atop{X+Y=Z}} 
\Phi\Big(\frac{X}{x}\Big) \Phi\Big(\frac{Y}{x}\Big) \Phi\Big(\frac{Z}{x}\Big),
\end{equation}
which counts both primitive and imprimitive solutions. We also set
\begin{equation}\label{304b}
N^{\ast}(x, y; \Phi) := \sum_{{X, Y, Z \in  \sS(y)}\atop{X+Y=Z, gcd(X,Y, Z)=1}} 
\Phi\Big(\frac{X}{x}\Big) \Phi\Big(\frac{Y}{x}\Big) \Phi\Big(\frac{Z}{x}\Big),
\end{equation}
which counts only primitive solutions. 
The asymptotic estimates below are obtained for a general $\Phi$.
For applications one may take
 $\Phi(x)$ to be  a nonnegative real-valued ``bump function'' supported
on $[0,1],$ which is equal to $1$ on $[\epsilon, 1- \epsilon]$
and is equal to $0$ outside $[\frac{\epsilon}{2}, 1- \frac{\epsilon}{2}]$.
This choice of weight function will ensure  that we essentially count solutions
$X+Y=Z$ in $y$-smooth numbers restricted to the range $[1,x]$.

In  \cite[Theorem 2.1]{LS09} we 
obtain  the following 
asymptotic formula, valid for  $\kappa >8$,
which counts both primitive and imprimitive
 weighted integer solutions.

%*****************************************************************
%  Theorem 4.1
%
%  Theorem 2.1 of [PLMS]
%
%*****************************************************************
\begin{theorem}~\label{th52}
{\em (Counting Weighted Smooth Integer Solutions)} 
Assume  the truth of the GRH.
Let $\Phi(x)$ be a fixed smooth compactly supported, real-valued 
weight function in  $ C_c^{\infty}(\RR^{+})$. 
Let $x$ and $y$ be large, with
$(\log x)^{8+\delta} \le y \le \exp( (\log x)^{1/2 - \delta})$
for some $\delta >0$. Define $\kappa$ by the relation $y= (\log x)^{\kappa}$.
 Then, we have
\begin{equation}\label{562}
N(x, y; \Phi) =\fS_{\infty}\Big(1-\frac{1}{\kappa}, \Phi\Big) \fS_f\Big(1-\frac{1}{\kappa}, y\Big) \frac{\Psi(x,y)^3}{x} +
O_{\eta}\left(  \frac{\Psi(x,y)^3}{x \log y}\right),
\end{equation}
Here the ``weighted archimedean singular series" (more properly,  ``singular integral")
is defined by
\begin{equation}\label{332}
{\fS}_{\infty}(c, \Phi):= c^3 \int_{0}^{\infty} \int_{0}^{\infty} 
\Phi(t_1)\Phi(t_2) \Phi(t_1+t_2)t_1 t_2(t_1+t_2)^{c-1} dt_1 dt_2,
\end{equation}
and the ``non-archimedean singular series" ${\fS}_{f}(c, y)$ is defined by
\begin{equation}\label{333}
{\fS}_{f}(c, y):= \prod_{p \le y} \Big( 1 + \frac{p-1}{p(p^{3c-1}-1)}
\Big(\frac{p-p^c}{p-1}\Big)^3 \Big)
\prod_{p>y}\Big( 1- \frac{1}{(p-1)^2}\Big).
\end{equation}
\end{theorem}

The compact support of $\Phi(x)$ away from $0$ guarantees that 
the ``weighted archimedean singular series" ${\fS}_{\infty}(c, \Phi)$ is
defined for all real $c$. The 
``non-archimedean singular series"  ${\fS}_{f}(c, y)$ is given
by an Euler product 
 which converges to an analytic function for $Re(c)> \frac{1}{3}$ and
diverges at $c= \frac{1}{3}$.
It has a phase change in its behavior as $y \to \infty$ at the threshold value $c=\frac{2}{3}$
corresponding to $\kappa=3$. Namely,  we have
\begin{equation}\label{421}
\lim_{y \to \infty} {\fS}_{f}\Big(1-\frac{1}{\kappa}, y\Big)= \left\{ 
\begin{array}{ll} {\fS}_{f}\Big(1-\frac{1}{\kappa}\Big) &~~\mbox{for}~~ \kappa>3,\\
~&~\\
+\infty &~~\mbox{for} ~~ 0<\kappa \le 3,
\end{array}
\right.
\end{equation}
where for $c> \frac{2}{3}$ we define
\begin{equation}\label{422}
\fS_{f}(c) := \prod_{p}\Big( 1 + \frac{p-1}{p(p^{3c-1}-1)} 
\Big(\frac{p - p^c}{p-1}\Big)^3 \Big).
\end{equation}
The Euler product  (\ref{422}) converges absolutely to an analytic function of
$c$ on the half-plane
 Re$(c) > \frac{2}{3},$ and  diverges at $c= \frac{2}{3}$. Outside this half-plane,
 on the range $\frac{1}{2} < c  \le \frac{2}{3}$, although one has
 ${\fS}_{f}(1-\frac{1}{\kappa}, y) \to \infty$ as $y \to \infty$,
 one can show that
 \[
 {\fS}_{f}\Big(1-\frac{1}{\kappa}, y\Big) \ll \exp \Big( y^{3/\kappa-1}  \Big).
 \]
 This  implies  that  for $2 < \kappa \le 3$ one has 
 $ {\fS}_{f}(1-\frac{1}{\kappa}, (\log H)^{\kappa})\ll H^{\epsilon}$ for any positive $\epsilon$,
 which suggests that the heuristic argument of section 1.2 may
 continue to apply to $N(H, \kappa)$ on this
 range.
 
For all $\Phi(x) \in C_c^{\infty}(\RR_{>0})$ 
the  ``weighted archimedean singular series" (\ref{332}) defines an
entire function of $c$. 
For the  special  
weight function $\Phi(x)$ being the step function $\chi_{[0,1]}(x)$
  which is $1$ on $[0,1]$, and $0$
elsewhere, the
``weighted archimedean singular series"  integral (\ref{332}) 
becomes the ``archimedean singular series" given in  (\ref{132}).
This weight function is not compactly supported on $(0, \infty)$
so is not covered by  Theorem~\ref{th52}. Indeed (\ref{132}) defines an
analytic function on the half-plane $Re(c) > \frac{1}{3}$,
which diverges when approaching $c=\frac{1}{3}$.

 To obtain  a bound for primitive solutions,
for a given range of $y$, 
we perform an inclusion-exclusion argument. 
 The following result applies to arbitrary compactly supported smooth test
functions (\cite[Theorem 2.2]{LS09}).

%**************************************************
% Theorem 4.2 
%
%  [Theorem 2.2 of [PLMS] (formerly thm 12.1)]
%
%**************************************************

\begin{theorem}~\label{th121}
{\em (Counting Weighted Primitive Integer Solutions)}
Assume the truth of the GRH.
Let $\Phi(x)$ be a fixed, smooth, compactly supported,
real valued function  in $ C_c^{\infty}(\RR^{+})$,
Choose any  $\delta>0$ and 
let $x$ and $y$ be large with $(\log x)^{8+\delta} \le y \le \exp( (\log x)^{1/2 - \delta}).$
Define $\kappa$ by the relation $y= (\log x)^{\kappa}$.
Then, we have
\[
N^{\ast}(x, y; \Phi)= 
{\fS}_{\infty}\Big(1-\frac{1}{\kappa}, \Phi \Big) 
{\fS}_{f}^{\ast}\Big(1-\frac{1}{\kappa}, y \Big) 
\frac{\Psi(x,  (\log x)^{\kappa})^3}{x}
+ O\Big( \frac{\Psi(x, y)^3}{x(\log y)^{\frac{1}{4}}} \Big),
\]
where the primitive non-archimedean singular series $\fS_{f}^{\ast}(c,y)$
was defined in (\ref{133}).
\end{theorem}

Theorem~\ref{th13} is an immediate consequence of Theorem~\ref{th121},
taking an appropriate limit of nonnegative weight functions $\Phi(x)$ approaching the characteristic
function  $\chi_{[0,1]}(x)$.

 The two theorems above imply the truth of  a weighted
analogue of Conjecture \ref{cj32} for $\kappa>8$, as follows.

 %*****************************************************************
% 
%  Theorem 4.3
%
%*****************************************************************

\begin{theorem}~\label{th35}
{\em (Relative Density of Weighted Smooth Solutions)}
Assume the truth of the GRH.
Then for any nonnegative real-valued function $\Phi(x) \in C_c^{\infty}(\RR_{>0})$
not identically zero,  there holds 
\[
 \lim_{x \to \infty} \,\frac{ N^{\ast}(x, (\log x)^\kappa; \Phi)}{N(x, (\log x)^{\kappa}; \Phi)} = 
 \,\frac{1}{\zeta (2- \frac{3}{\kappa})}~~\mbox{for}~~ \kappa>8.
 \]
\end{theorem}

\paragraph{Proof.}
This result is based on the identity of Euler products
 \begin{equation}\label{423}
 \fS_{f}^{\ast}(c)
 := \prod_{p}\Big( \Big( 1 + \frac{p-1}{p(p^{3c-1}-1)} 
\left(\frac{p - p^c}{p-1}\right)^3 \Big) \Big(1- \frac{1}{p^{3c-1}}\Big)\Big)= \frac{1}{\zeta(3c- 1)}\fS_{f}(c).
\end{equation}
This identity shows that $\fS_{f}(c)$ has a meromorphic continuation to the half-plane
$Re(c)> \frac{1}{2}$, with its only singularity on
this region being a simple pole at $c=\frac{2}{3}$ having residue 
%[V13 corrected one line here]
$\frac{1}{3}\fS_{f}^{\ast}(\frac{2}{3})$.
In particular, for real $c = 1 -\frac{1}{\kappa}>\frac{2}{3}+\epsilon$ we have
$$
\fS_{f}(c, y) = \fS_f(c) \Big(1+ O_{\epsilon}\Big(\frac{1}{y}\Big)\Big),
$$
and for real $c > \frac{1}{2} +\epsilon$ we have
$$
\fS_{f}^{\ast}(c, y) = \fS_f^{\ast}(c) \Big(1+ O_{\epsilon}\Big(\frac{1}{y}\Big)\Big).
$$

 Using these estimates in the main terms of Theorem~\ref{th52} and  Theorem~\ref{th121}, 
 yields  for $\kappa> 8 +\delta$, the estimate  
\begin{equation}\label{1202b}
N^{\ast}(x, (\log x)^{\kappa}; \Phi)=
\frac{1}{\zeta(2-\frac{3}{\kappa})}N(x, (\log x)^{\kappa}; \Phi))
\left( 1+ O_{\delta}
\left( \frac{1}{(\log\log x)^{\frac{1}{4}}}\right)\right).
\end{equation}
The positivity hypothesis  on $\Phi(x)$ implies that $N(x, (\log x)^{\kappa}; \Phi)>0$
so we may divide both sides of (\ref{1202b}) by it to obtain the ratio estimate (\ref{375}).
$~~\Box$

%***************************************************
%
%
% Section 5
%
%
%***************************************************

\section{ Proof Sketch for Theorem \ref{th52}}

 Theorem~\ref{th52} is established in \cite{LS09}
using  the Hardy-Littlewood method (\cite{HL23}, \cite{HL24}), in the modern form using
finite exponential sums, see Vaughan \cite{Va97}.
We introduce the  {\em weighted exponential sum}
\[
E(x, y; \alpha) := \sum_{n \in \sS(y)} e(n \alpha) \Phi(\frac{n}{x}),
\]
where $e(x) := e^{2\pi i x}$. We have the identity
\begin{equation}\label{306}
N(x, y; \Phi)= 
\int_{0}^{1} E(x,y; \alpha)^2 E(x,y; - \alpha) d \alpha,
\end{equation}
because in multiplying out the exponential sums in the integral,
only terms $(n_1, n_2, n_3)$ with $n_1+n_2-n_3=0$ contribute.
The Hardy-Littlewood method  estimates the integral on the
right side of (\ref{306}) by splitting the integrand into
small arcs centered around rational points with small
denominators, and adding up the contributions of
the arcs.  The major contribution will come from those parts of the
circle very 
near points $\frac{a}{q}$ with small denominator, the {\em major arcs}.
The remainder of the circle comprises the {\em minor arcs}. Our choice
of major arcs and minor arcs  is given below. \\

%***************************************************
%
%
% Section 5.1
%
%
%***************************************************

\subsection{General bound for $E(x, y, \alpha)$}

To estimate the integral (\ref{306}) we  wish to determine
how the function $E(x,y, \alpha)$ behaves for $\alpha$
near a rational number $\frac{a}{q}$ in lowest terms, and we write 
$\alpha= \frac{a}{q}+ \gamma.$ The main  estimate for these
is given by \cite[Theorem 2.3]{LS09}.

%*********************************************
%  Theorem 5.1
%   [Theorem 2.3 of PLMS]
%********************************************
\begin{theorem}\label{thm51}
Assume the truth of the GRH.  Let $\delta>0$ be any 
fixed real number.  Let $x$ and $y$ be 
large with $(\log x)^{2+\delta} \le y\le \exp((\log x)^{1/2 - \delta})$, and let 
$\kappa$ be defined by $y=(\log x)^{\kappa}$.  Let $\alpha\in [0,1]$ be a 
real number with $\alpha=a/q+\gamma$ where $q\le \sqrt{x}$, $(a,q)=1$, and 
$|\gamma|\le 1/(q\sqrt{x})$. 

(1)  If $|\gamma|\ge x^{\delta-1}$ then for any fixed $\epsilon>0$,  we have 
\[
E(x,y;\alpha) \ll x^{\frac 34+\epsilon}. 
\] 

(2) If $|\gamma|\le x^{\delta-1}$ then on writing $q=q_0q_1$ with 
$q_0 \in {\sS(y)}$ and all prime factors of $q_1$ being bigger than $y$, and writing ${c_0}=1-1/\kappa$, 
then for any fixed $\epsilon>0$ we have
\begin{eqnarray*}
E(x,y;\alpha)&=&\frac{\mu(q_1)}{\phi(q_1)} 
\frac{c_0}{q_0^{c_0}}\prod_{p|q_0} \Big(1- \frac{p^{c_0}-1}{p-1}\Big) 
\Big(\int_0^{\infty} 
\Phi(w) e(\gamma xw) w^{c_0} \frac{dw}{w}\Big) \Psi(x,y) \\
&&~~+ ~O\Big(x^{\frac 34+\epsilon}\Big) 
+ O\Big( \frac{\Psi(x,y)q_0^{-c_0+\epsilon} q_1^{-1+\epsilon}}{(1+|\gamma|x)^2} \frac{(\log \log y)}{\log y} \Big). \\
\end{eqnarray*}
\end{theorem}

These bounds are obtained combining the Hildebrand-Tenenbaum saddle point
method  with bounds for partial Euler products, as we now explain. This estimate explicitly
contains $\Psi(x, y)$ in main term, avoiding the problem that behavior of $\Psi(x, y)$ for
very small values of $y$ does not have a convenient simplifying
formula. 
%***************************************************
%
%
% Section 5.2
%
%
%***************************************************

\subsection{Bounding partial Euler products}

The Dirichlet series associated to
the set  $\sS(y)$ of all integers having smoothness bound $y$ is given
by the partial Euler product
\[
\zeta(s; y) =  \prod_{p \le y} (1- p^{-s})^{-1} = \sum_{n \in \sS(y)} n^{-s},
\]
associated to the Riemann zeta function; this series converges
absolutely on the half-plane $Re(s)>0$.
We invoke the GRH to control the size of $\zeta(s; y)$
and more generally for  the partial Euler products
associated to primitive Dirichlet $L$-functions,
\[
L(s, \chi; y):= \prod_{p \le y} \Big( 1- \chi(p) p^{-s}\Big)^{-1},
\]
(\cite[Proposition 5.1]{LS09}). 

%%%%%%****************
%
% Proposition 5.2 
%%  [Prop. 51 of PLMS]
% 
%%%%%%%%%***********
\begin{proposition}~\label{le43}  Assume  the truth of the GRH.  Let $\chi \bmod{q}$ be a 
primitive Dirichlet character.       
For any $\epsilon >0$, and $s$ a complex number with Re$(s)=\sigma\ge 1/2+\epsilon$, we 
have  
\[
|L(s, \chi; y)| \ll_{\epsilon} (q|s|)^{\epsilon}.
\]
For the trivial character we have, with $\sigma ={\rm Re}(s)\ge 1/2+\epsilon$,  
\[
|\zeta(s; y)| \ll_{\epsilon} \exp\Big(\frac{y^{1- \sigma}}{(1+|t|)\log y} \Big) |s|^{\epsilon}.
\]
\end{proposition}

This result is analogous to a Lindel\"{o}f hypothesis bound.
It is proved by using the GRH and the ``explicit formula'' techniques 
of prime number theory (see Davenport \cite[Chaps. 17 and 19]{DM80})
to estimate $\sum_{n\le y} \Lambda(n)\chi(n) n^{-it}$,
followed by partial summation to estimate $\log |L(\sigma+it, \chi; y)|$.   
%***************************************************
%
%
% Section 5.3
%
%
%***************************************************

\subsection{Estimating the weighted exponential sum}

We obtain the following estimates for the exponential sum
at a value $\alpha= \frac{a}{q} + \gamma$ near a rational point (\cite[Proposition 6.1]{LS09}).
%*****************************************************************
%
% Proposition 5.3
%
% [Prop. 6.1 of PLMS]
%  
%*****************************************************************
\begin{proposition} \label{pr61}  Assume the truth of the GRH.  Let $\alpha$ be a 
real number in $[0,1]$ and write $\alpha =a/q + \gamma$ with $(a,q)=1$, 
$q\le \sqrt{x}$, and $|\gamma| \le 1/(q\sqrt{x})$.   Then 
\[
E(x,y;\alpha) = M(x,y;q,\gamma) + O(x^{\frac{3}{4}+\epsilon}), 
\]
where the ``local main term" $M(x,y;q, \gamma)$ is defined by 
\[
M(x,y; q, \gamma) := \sum_{n \in \sS(y)} 
\frac{\mu(\frac{q}{(q,n)})}{\phi(\frac{q}{(q,n)})}e(n \gamma) \Phi\Big(\frac{n}{x}\Big).
\]
\end{proposition} 

This result is proved using an expansion in terms of Dirichlet characters
\[
E(x, y;  \alpha)=
\sum_{{d|q}\atop{d \in \sS(y)}} \frac{1}{\phi(q/d)} 
\sum_{\chi\bmod{q/d}} \chi(a) \tau(\bar{\chi}) 
\sum_{m \in \sS(y)} e(md \gamma) \chi(m)
\Phi\Big(\frac{md}{x}\Big).
\]
The contribution of the principal characters to this sum gives the 
``local main term" above. 
The contribution of each non-principal character is 
shown to be bounded by 
\[
\frac {\sqrt q}{\sqrt d}\sum_{m \in \sS(y)} e(md \gamma) \chi(m) 
\Phi\Big(\frac{md}{x}\Big) \ll x^{\frac 34+\epsilon},
\]
This  is established using Proposition \ref{le43} for 
primitive characters, with an additional observation to handle
imprimitive characters.

%***************************************************
%
%
% Section 5.4
%
%
%***************************************************

\subsection{Estimating the ``local main terms"}

To study the ``local main term" we first observe that 
we may uniquely factor any integer
 $q=q_0q_1$ where $q_0$ is divisible only by primes at most $y$, and $q_1$ is divisible 
only by primes larger than $y$. Then we have the identity
\begin{equation}\label{621a}
M(x,y;q,\gamma) = \frac{\mu(q_1)}{\phi(q_1)} M(x,y;q_0,\gamma).
\end{equation}
Thus we reduce to studying the local main term $M(x,y;q_0,\gamma)$
with $q_0 \in \sS(y)$. Here 
  use the Hildebrand-Tenebaum
saddle-point method, expressing it as a contour integral. 
We have, for $\sigma=Re(s) >1$, 
\begin{equation}\label{561a}
M(x,y;q_0,\gamma) =  \frac{1}{2\pi i} \int_{\sigma-i\infty}^{\sigma+i\infty} 
\zeta(s;y)H(s;q_0) x^s {\check \Phi}(s,\gamma x) ds.  
\end{equation}
in which 
\[
H(s;q_0) = q_0^{-s} \prod_{p|q_0} (1- \frac{p^s-1}{p-1}),
\]
and
\[
\check{\Phi}( s, \lambda)
:= \int_{0}^{\infty} \Phi(w)e(\lambda w) w^{s-1}dw.
\]
The Hilbert-Tenenbaum saddle point method applied to $\zeta(s; y)$  deforms the contour (\ref{561a})
 to approximate near the real axis  the vertical line $Re(s)= c$
where $c= c(x,y)$ is the {\em Hildebrand-Tenenbaum saddle point value}. This is
defined to be the
unique positive solution of the equation
\[
\sum_{p \le y} \frac{\log p}{p^c -1} = \log x.
\]
The root is unique because the function
\[
g(c;y) := \sum_{p \le y} \frac{\log p}{p^c-1}
\]
is strictly decreasing for $c>0$, with limit  $+\infty$ as $c \to 0^{+}$
and limit $0$ as $c \to \infty$.  We use this saddle-point value in the integral (\ref{561a}) and
obtain  \cite[Proposition 6.2]{LS09}.
%*****************************************************************
% Prop. 5.4
% [Prop. 6.2 of PLMS]
% 
%
%*****************************************************************
\begin{proposition}\label{pr62}  Let $x$ and $y$ be large, and assume that $(\log x)^{2+\delta} \le y\le 
\exp((\log x)^{1/2 - \delta})$. Let $c=c(x,y)$ denote the 
 Hildebrand-Tenenbaum saddle point value.
Suppose $q_0 \in \sS(y)$
and let $\gamma$ be real with $|\gamma| \le 1/(q_0\sqrt{x})$,  
and let $M(x,y;q_0,\gamma)$ be as in Proposition \ref{pr61}.   
Then we have:

(1) If $|\gamma| \ge x^{\delta -1}$ then for any fixed $\epsilon>0$, we have
\[
|M(x,y;q_0, \gamma)| \ll x^{\frac{3}{4}+ \epsilon}q_0^{-\frac 34+\epsilon}.
\] 

(2) If $|\gamma| \le x^{\delta-1}$  then for any fixed $\epsilon>0$, we have 
\begin{eqnarray*}
 M(x,y;q_0,\gamma)& =&   \frac{1}{q_0^c} \prod_{p|q_0} 
\Big( 1-\frac{p^c-1}{p-1}\Big) (c{\check \Phi}(c,\gamma x)) \Psi(x,y) + O_{\epsilon}(x^{\frac 34+\epsilon}q_0^{-\frac 34+\epsilon})\\
&&\hskip 1 in + O_{\epsilon}
\Big(\frac{\Psi(x,y) q_0^{-c+\epsilon}}{(\log y)(1+|\gamma| x)^2}\Big).
\end{eqnarray*}
\end{proposition}

We have the formula, valid for $\kappa\ge 1+ \delta,$
\begin{equation}\label{642}
c(x,y)= 1-\frac{1}{\kappa} + O_{\delta}\left( \frac{\log \log y}{\log y}\right),
\end{equation}
deducible from \cite[Theorem 2]{HT86}. This estimate
 is used in replacing $c$ by $c_0$ in Theorem \ref{thm51} above.

%***************************************************
%
%
% Section 5.5
%
%
%***************************************************

\subsection{Hardy-Littlewood method estimates}

For parameter values $(x,y)$ we dissect the unit interval in
the integral (\ref{306}) into Farey arcs 
depending only on the parameter $x$. These are enumerated by
rational numbers in the Farey sequence $\sF(Q)$
of order $Q= x^{\frac{1}{2}}$. Here
$\sF( x^{\frac{1}{2}}) := 
\{ \frac{a}{q}: (a, q)=1 ~~\mbox{and}~~ 1 \le q \le  x^{\frac{1}{2}} \}.$
The Farey interval assigned to $\frac{a}{q}$ is the arc
between it and the mediant $\frac{a+a'}{q+q'}$ of its left
neighbor $\frac{a'}{q'}< \frac{a}{q}$, and the same for its right neighbor;
these intervals partition the unit interval.
We extract from  some of  these Farey intervals  the {\em major arcs}.
The  definition 
of the major arcs depends  on 
 an initially specified  {\em  cutoff parameter}  $\delta$ 
satisfying  $0< \delta \le \frac{1}{4}.$ 
Decreasing this parameter makes the major arcs smaller.
We will eventually let $\delta$ become arbitrarily small.
The set of {\em major arcs} $\fM$ consists of an interval associated
with each $\frac{a}{q}$ with $1 \le q \le x^{\frac{1}{4}}$, 
        which consists of that
       part of the Farey interval of $\frac{a}{q}$ satisfying
\[ 
\fM(\frac{a}{q}) = \Big\{ \alpha:~|\alpha- \frac{a}{q}| \le x^{\delta -1}\Big\}.
\]
The set of {\em minor arcs} $\fm$ consist of   the rest of the
interval $[0,1]$ not covered by the major arcs.

We do not use all of the Farey intervals in $\sF(x^{\frac{1}{2}})$ as major
arcs, but we carry out some intermediate estimates for all such intervals, for
possible later uses.

The estimates of Proposition \ref{pr61}(1) and Proposition \ref{pr62}(1) with
(\ref{621a})  lead to a bound $O_{\epsilon}\Big( x^{\frac{3}{4} + \epsilon}\Psi(x,y)\Big)$
for the minor arcs contribution.
Next Proposition~\ref{pr62}(2) is used to simplify the integral over the major arcs to 
\begin{eqnarray*}N(x, y; \Phi) & \approx &
\sum_{{1 \le q_0 \le x^{1/4}}\atop{q_0 \in \sS(y)} } \sum_{{a=1}\atop{(a, q)=1}}^q
\int_{-x^{\delta-1}}^{x^{\delta-1}} E(x, y; q_0, \frac{a}{q} + \gamma)^2 M(x,y; q_0; -\frac{a}{q} -\gamma) d\gamma,  \\
&  \approx   & \sum_{{1 \le q_0 \le x^{1/4}}\atop{q_0 \in \sS(y)} } 
\phi(q) \int_{-x^{\delta-1}}^{x^{\delta-1}} M(x, y; q_0, \gamma)^2 M(x,y; q_0; -\gamma) d\gamma,
%\label{531f}
\end{eqnarray*}
on noting that all  fractions  $\frac{a}{q}$ contribute identical main terms to
the latter integral, 
with the  minor arcs estimates used to control the remainder terms.

We then estimate $M(x, y; q , \gamma)$ in the formula above using Proposition \ref{pr62}(2).
It remains to integrate out the
$\gamma$-variable in 
this formula over the major arcs, which produces the archimedean singular series
contribution $\fS_{\infty}(c, \Phi)$ appearing in the final answer in Theorem~\ref{th52},
and the sum over $q=q_0q_1$ in the major arcs terms in (\ref{621a})
, which is multiplicative, produces the 
non-archimedean singular series contribution $\fS_{f}(c, y)$ appearing in
Theorem \ref{th52}. 
The difference between the contribution of  $q_0 \in \sS(y)$ and that of 
$q_1$ divisible only by primes greater than $y$
 in (\ref{621a}) accounts for the difference in Euler product
factors for $p \le y$ and $p > y$ in the formula (\ref{133}) for  $\fS_{f}^{*}(y, c)$.
Finally the estimate (\ref{642}) allows the replacement of $c$
by $1-\frac{1}{\kappa}$ in the two singular series, to prove Theorem ~\ref{th52}. 

Finally we note that
the  inclusion-exclusion
argument that produces Theorem ~\ref{th121}
leads to the ``primitive non-archimedean singular series"
${\fS}_{f}^{\ast}(1-\frac{1}{\kappa}, y)$.

%***************************************************
%
%
% Subsection 6
%
%
%***************************************************
\section{Concluding Remarks}

We discuss various complementary questions and related problems.\\

(1) {\em Coefficients and  side congruence constraints.}
The  Hardy-Littlewood method approach for the upper bound 
should apply  to  other  linear  additive
problems involving smooth numbers.  One can certainly treat smooth solutions  of
homogeneous linear ternary Diophantine equations having arbitrary coefficients $(a, b, c)$, i.e.
$$
aX +bY+cZ=0.
$$
The singular series will need to be  modified appropriately.
One could also impose congruence side conditions, on the allowable prime factors. For 
example, one could consider  smooth solutions with 
all prime factors
$p \equiv 1~(\bmod ~4)$. In this situation there may occur local 
congruence obstructions to existence of solutions.\\

(2) {\em Additional variables.}
One can also treat by this method smooth solutions to 
linear homogeneous Diophantine equations
in more variables, i.e. 
$X_1 +X_2 + \cdots + X_n=0,$
for $n \ge 4$. Here it is natural to restrict to primitive solutions which also have
the non-degeneracy property that no sum of a proper subset of the variables
vanishes. \\

The heuristic given in Section 1.2 is easily modified to show
 that the cutoff smoothness bound
for infinitely many non-degenerate primitive solutions should be
$$
S(X_1, X_2, ..., X_n) \ll \left(\log H(X_1, ..., X_n)\right)^{1+ \frac{1}{n-1} + \epsilon}.
$$
Here we can again add coefficients $\sum a_j X_j$ and put congruence restrictions
on the allowed prime divisors of the $X_i$.\\

(3) {\em  Binary additive variant.}
Consider the (inhomogeneous) binary equation
\begin{equation}\label{501}
X+Y=1.
\end{equation}
Here the relative primality of solutions is built in to the equation.
%A similar heuristic to  the introduction applies to estimate 
%how small might be smooth solutions
%to this equation. 
%It predicts:  {\em For any fixed $\epsilon >0$, there will be infinitely many solutions 
%$(X, Y, 1)$ satisfying the 
%smoothness bound 
%$$
%S(X, Y, 1) \ll (\log H)^{2+\epsilon}.
%$$
%}
A similar heuristic to that in \S 1.2  applies to 
indicate there should be  a threshold value $\kappa_0^{\ast}$ such that there are
infinitely many  solutions $(X, 1, X+1)$ having 
smoothness
$S(X, 1, X+1) \ll (\log X)^{\kappa_0^{\ast}+ \epsilon}$, but only finitely
many solutions having $S(X, 1, X+1) \ll (\log X)^{\kappa_0^{\ast} - \epsilon}$.
This heuristic predicts that  $\kappa_0^{\ast} =2$.
The examples in \S 1.3 include   several ``unusually good" triples with
$\kappa_{0}( X, 1, X+1) < 2.$ 
However  obtaining lower bounds  for the number of solutions to (\ref{501}), even
conditionally under GRH,  seems out of reach.
The circle method is unable to effectively control minor arc estimates for binary problems.\\

(4) {\em Extension to algebraic number fields.}
The $\ABC$-conjecture has been generalized to number fields.
For the latest results in this direction see Masser \cite{Ma02},
Gy\H{o}ry and Yu \cite{GY06}
and Gy\H{o}ry \cite{Gy08}.
One might formulate  an extension of  the $\XYZ$-conjecture to algebraic number fields, considering
triples $(X, Y, Z)$ of algebraic integers with $X+Y+Z=0$.
In this case the smoothness bound will involve the (absolute) norms of  prime ideals dividing the
algebraic integers in the equation, and should be scaled to be independent of the
field of definition of the equation. \\

(5) {\em Reverse implications for zeros of $L$-functions.}
It is known that a uniform generalization  of the $\ABC$-conjecture for number fields has  implications concerning location of zeros of $L$-functions,
at least concerning nonexistence of ``Siegel zeros"
cf. Granville and Stark \cite{GS00}. Their results are based on analyzing
algebraic integer solutions to modular equations, cf. (\ref{108a}).
Since the lower bound for the $\XYZ$-conjecture is
a weaker implication than the $\ABC$-conjecture, one may ask whether a suitable
uniform $\XYZ$-conjecture for number fields is sufficient to imply  nonexistence of
Siegel zeros, by a similar approach. \\

(6) {\em Function field case}. Let $K$ be a function field of one variable over
a suitable non-algebraically closeld $k$, e.g $k = \QQ$ or $k= \FF_q$.
 In this context the $\ABC$-conjecture is an unconditional theorem, and the GRH 
 is unconditional over finite fields.
 It seems plausible that some analogues of the $\XYZ$ conjecture
 might be true and provable unconditionally in this context.

%***************************************************
%
%
% References
%
%
%***************************************************

\end{document}